\documentclass{article}

\usepackage{amsmath}
\usepackage{amssymb}
\usepackage[english]{babel}
\usepackage{amsthm}
\usepackage[latin1]{inputenc}
\usepackage{hyperref}
\usepackage{graphicx}
\usepackage{subfigure}
\usepackage{pb-diagram}

\newtheorem{teo}{Theorem}
\newtheorem{lemma}{Lemma}
\newtheorem{prop}{Proposition}
\newtheorem{defin}{Definition}
\newtheorem{cor}{Corollary}

\newenvironment{sistema}%
{\left\lbrace\begin{array}{@{}l@{}}}%
{\end{array}\right.}

{\left( \begin{array}{@{}l@{}}}%
{\end{array}\right)}

\begin{document}

\title{Chaotic dynamics in an impact problem}
\author{\textbf{Stefano Marò} \\
\textit{\small{Dipartimento di Matematica - Università di Torino}}\\
\textit{\small{Via Carlo Alberto 10, 10123 Torino - Italy}}\\
\textit{\small{e-mail: stefano.maro@unito.it}}
}

\date{}
 \maketitle
\begin{abstract}
We consider the model describing the vertical motion of a ball falling with constant acceleration on a wall and elastically reflected. The wall is supposed to move in the vertical direction according to a given periodic function $f$. We show that a modification of a method of Angenent based on sub and super solutions can be applied in order to detect chaotic dynamics. Using the theory of exact symplectic twist maps of the cylinder one can prove the result under "natural" conditions on the function $f$.  
\end{abstract}


\section{Introduction}
Mechanical models with impacts appear in many engineering applications, such as the modeling of pneumatic hammers or various machinery with moving parts. This kind of models also appear in Theoretical Physics, for example the Fermi-Ulam oscillator provides a model for the motion of particles between two galaxies. It is not surprising that this topic has been widely studied both from the analytical and numerical point of view. We cite the monograph \cite{brogliato} and the references therein for a good insight.

In this paper we concentrate on a particular impact problem. We consider the model of a free falling ball on a moving racket. The racket is supposed to move periodically in the vertical direction according to a function $f(t)$ and the ball is reflected according to the law of elastic bouncing when hitting the racket. The only force acting on the ball is the gravity $g$. Moreover, it is usual to assume that the mass of the racket is huge with respect to the mass of the ball. It means that the impacts do not affect the motion of the racket.
  
A good strategy to describe the motion of the ball is to define a map $P$ that sends a couple $(t_0,v_0)$ representing the time of impact and the velocity immediately after it to the next impact time and corresponding velocity $(t_1,v_1)$. In order to write the map,  Holmes \cite{holmes} considered the approximation given by the assumption of a large amplitude of the motion of the ball with respect to the amplitude of the motion of the racket. As an example, he considered the case $f(t)=\beta\sin\omega t$. In this case the model is the one given by the so-called standard map
\begin{equation*}
\begin{sistema}
t_1=t_0+\frac{2}{g}v_0 \\
v_1=v_0+2\beta\omega\cos\omega t_1.
\end{sistema}
\end{equation*}
This map is area preserving and, for $v_0$ sufficiently large, is defined on the cylinder. It is widely studied and complex dynamics appear when the amplitude $\beta$ is sufficiently large. Holmes approximation suggests that this very simple mechanical model shows an interesting dynamics.\\
We are going to study the exact model, without Holmes approximation. In this case one is lead to the following map
\begin{equation*}
P:
\begin{sistema}
t_1=t_{0}+\frac{2}{g}v_{0}-\frac{2}{g}f[t_1,t_{0}]+\frac{2}{g}\dot{f}(t_0)
\\
v_1=v_{0}+\dot{f}(t_1)-2f[t_1,t_{0}]+\dot{f}(t_0)
\end{sistema}
\end{equation*}
where
$$
f[t_1,t_{0}]=\frac{f(t_1)-f(t_{0})}{t_1-t_{0}}.
$$
This is the map (up to changes of systems of references) considered by Pustyl'nikov \cite{pustsoviet, pust}. He proved that if the motion of the racket is periodic and there exists an instant in which the velocity is sufficiently large, precisely if
\begin{equation*}
\max\dot{f}\geq\frac{g}{2},
\end{equation*}  
then unbounded motions of the ball are possible. Adding some hypothesis on $f$ he was able to prove that there exists a set of positive measure in the phase-space $(t,v)$ leading to unbounded motions. Nonetheless bounded motions also exist. Actually, in \cite{marobounce} we proved that if $f$ is regular then for every real number $\omega$ sufficiently large, there exists a solution with rotation number $\omega$.  Among the huge amount of results concerning such model we cite \cite{desimoi, jiang, luohan, okn, ruiztorres}. It is worth mentioning also the paper by Dolgopyat \cite{dolgo} dealing with non-gravitational potentials and the paper by Kunze and Ortega \cite{kunzeortega2} dealing with non-periodic functions $f$.\\
Holmes results and the coexistence of bounded and unbounded motions, together with many numerical evidences, motivated us to the analytical study of the chaotic behaviour of the model. Chaotic dynamics are understood as the existence of a compact invariant set $\mathcal{K}$ such that some iterate of the map restricted to it is semi-conjugated to the Bernoulli shift. This is a classical definition of chaos and is used in many contexts. See \cite{kircstof,palmer} for more details. On this line, in a recent paper, Ruiz-Herrera and Torres \cite{ruiztorres} considered the application of a general topological tool based on stretching techniques. In this way, they constructed some particular periodic functions $f$ for which the corresponding model of bouncing ball shows chaotic dynamics. Note that, as in Holmes approach, their result is valid for a particular choice of the function $f$. We will show how a different approach, based on symplectic techniques, can give the result for a large class of functions characterized by a "natural" condition.    

It can be shown that if we pass from the variables time-velocity to the variables time-energy through the change $E=\frac{1}{2}v^2$ then $P$ is an exact symplectic twist map of the cylinder. This kind of maps have been widely studied and one can look at \cite{angenent, bangert,  matherams} to have a deep insight. Our result makes use of this theory. Precisely we are going to prove the following. If $f\in C^2(\mathbb{R})$ is $1$-periodic and such that
\begin{equation}\label{thu}
\max\dot{f}\geq\frac{g}{2} \quad\mbox{and}\quad \min\dot{f}\leq -\frac{g}{2},
\end{equation}          
then the model presents infinitely many invariant sets $\mathcal{K}$ with chaotic dynamics. Here we understand that there exist infinitely many disjoint invariant compact sets $\mathcal{K}_i$ and for every such set, some iterate of the map restricted to it is semi-conjugated to the Bernoulli shift. 

Condition (\ref{thu}) is the same considered by Pustyl'nikov to construct unbounded solutions. This latter fact is not casual. Indeed it comes from the work of Angenent \cite{angenent} that, for an exact symplectic twist map of the cylinder, the presence of an unbounded orbit in the future and an unbounded orbit in the past implies chaos.  

The orbits $(t_n,E_n)$ of the map $P$ can be characterized as the solutions of a recurrence relation for the sole variable $t_n$. More precisely, one can define the recurrence relation 
\begin{equation}\label{ricu}
\partial_{2}h(t_{n-1},t_{n})+\partial_{1}h(t_{n},t_{n+1})
\end{equation}
where the function $h$ is called generating function and characterizes exact symplectic twist maps. We have that $(t_n,E_n)$ is an orbit of $P$ if and only if
\begin{equation*}
\partial_{2}h(t_{n-1},t_{n})+\partial_{1}h(t_{n},t_{n+1})=0\quad\mbox{for every }n
\end{equation*}
and $E_n=\partial_{1}h(t_n,t_{n+1})$. Angenent considered general recurrence relations $\Delta(t_{n-1},t_n,t_{n+1})$ with similar properties as (\ref{ricu}). Supposing the existence of suitable sub and super solutions he was able to build special sequences $(t_n)$ satisfying 
$$
\Delta(t_{n-1},t_n,t_{n+1})=0 \quad\mbox{for every } n
$$
and from which one can detect chaos. This method cannot be applied directly to our situation. In fact, one needs $\Delta$ to be defined in the whole $\mathbb{R}^3$. It means that the generating function $h$ should be defined on the whole $\mathbb{R}^2$ (and, equivalently, the map $P$ to be defined on the whole cylinder). The latter condition is not satisfied as the map $P$ is defined only for large values of the energy. We will overcome this problem extending the map $P$ to the whole cylinder in a way such that Angenent method can be applied. Actually, it may occur that the invariant sets generating chaos stay in the part of the cylinder in which the original map $P$ is not defined. These sets would be meaningless for our problem. The heart of the paper is to prove that one can choose suitable sub and super solutions such that the corresponding sets $\mathcal{K}_i$ are contained in the part of the cylinder in which $P$ and $\tilde{P}$ coincide. This will produce the desired chaos for the original map $P$.\\
We conclude this introductory section saying that condition \ref{thu} is not optimal as it is reasonable that chaos might appear for a generic $f$. On the other hand one cannot expect chaos for every $1$-periodic function $f(t)$. For example, in the trivial case $f\equiv 0$ one gets the integrable twist map that does not show chaotic behaviour. 

The rest of the paper is organized as follows. In Section \ref{statement} we give a precise statement of the problem and of the result we want to prove. In Section \ref{modified} we construct the extension of the generating function and in Section \ref{subsuper} we construct the sub and super solution. Finally, Section \ref{prof} is dedicated to show that the sub and super solution that we constructed are the one from which we can get the result. This is achieved through a detailed revision of Angenent proof.

\section{Statement of the problem}\label{statement}
We are concerned with the problem of a bouncing ball on a moving racket in the vertical direction. The motion of the racket is supposed to be $1$-periodic and described by a function $f\in C^2(\mathbb{R})$. To study the motion of the ball, the approach of Pustylinkov \cite{pustsoviet} or Kunze-Ortega \cite{kunzeortega2} was to consider a discrete map investigating the behaviour of the sequences of impact times on the racket $(t_n)$ and the corresponding velocity immediately after the impact $(v_n)$. The map considered by Kunze and Ortega was the following: 
\begin{equation}
P:
\begin{sistema}\label{mapp}
t_1=t_{0}+\frac{2}{g}v_{0}-\frac{2}{g}f[t_1,t_{0}]+\frac{2}{g}\dot{f}(t_0)
\\
v_1=v_{0}+\dot{f}(t_1)-2f[t_1,t_{0}]+\dot{f}(t_0).
\end{sistema}
\end{equation}
In \cite{marobounce} one can see how to get the map $P$ starting from the associated differential equation describing the motion of the ball. The map $P$ is well defined for $v_0>\bar{v}$ for some $\bar{v}$ sufficiently large. So one can say that $P$ is a $C^2$ map of the cylinder $\mathbb{T}\times (\bar{v},+\infty)$, where $\mathbb{T}=\mathbb{R}/\mathbb{Z}$. Note that the condition $v_0>\bar{v}$ corresponds to the existence of a large $K_1>0$ such that  
\begin{equation*}
t_1-t_0>K_1.
\end{equation*}
This map expressed in the variables $(t_n,v_n)$ is not exact symplectic but $P(t_0,E_0)\mapsto(t_1,E_1)$ where $E_0:=\frac{1}{2}v_0^2$ is a $C^2$ exact symplectic map of the cylinder $\mathbb{T}\times (\bar{E},+\infty)$ where $\bar{E}$ is sufficiently large. The coordinates $(t_n,E_n)$ are conjugate, so the map can be expressed in terms of a $C^3$ function $h(t_0,t_1)$ defined for $t_1-t_0>K_1$ such that
\begin{equation*}
\begin{sistema}
\partial_1h(t_0,t_1)=E_0 \\
\partial_2h(t_0,t_1)=-E_1.
\end{sistema}
\end{equation*}
The function $h$ is called generating function, satisfies the periodicity condition
$$
h(t_0+1,t_1+1)=h(t_0,t_1)
$$ 
and can be explicitly computed. One can see that a sequence $(t_n,E_n)$ is an orbit of $P$ if and only if
\begin{equation}\label{statt}
\partial_{2}h(t_{n-1},t_{n})+\partial_{1}h(t_{n},t_{n+1})=0\quad\mbox{for every }n
\end{equation}
and $E_n=\partial_{1}h(t_n,t_{n+1})$ for every $n$. Sequences satisfying (\ref{statt}) are called $h$-stationary configurations. See \cite{kunzeortega2} for more details.

With this setting we shall prove 
\begin{teo}\label{maint}
Let $f$ be $1$-periodic and of class $C^2$. Suppose that
\begin{equation}\label{hpf}
\max\dot{f}\geq\frac{g}{2} \quad\mbox{and}\quad \min\dot{f}\leq -\frac{g}{2}.
\end{equation} 
Then there exist a sequence of compact sets $\mathcal{K}_i\subset\mathbb{T}\times(\bar{E},+\infty)$ and a sequence of positive integers $Q_i$ such that $P^{Q_i}$ leaves $\mathcal{K}_i$ invariant and is semi-conjugated to a Bernoully shift when restricted to $\mathcal{K}_i$.
\end{teo}



\section{The modified generating function}\label{modified}

We remember that the generating function $h$ is $C^3$ and defined for $t_1-t_0>K_1$. Moreover, one could prove (\cite{marobounce}) that there exist $K\geq K_1$ and $\delta>0$ such that $\partial_{12}h>\delta$ for $t_1-t_0>K$. Now extend $\partial_{12}h$ in the region $t_1-t_0\leq K$ as a $C^1$ bounded function with lower bound given by $\delta$. Consider $\psi$ a $C^\infty$ non-decreasing cut-off function of $\mathbb{R}$ such that
\begin{equation*}
\begin{sistema}
\psi=1 \mbox{ if }t_1-t_0>K-\epsilon/4 \\
\psi=0 \mbox{ if } t_1-t_0<K-3\epsilon/4.
\end{sistema}
\end{equation*}
Now let $\chi(t_0,t_1)=\psi(t_1-t_0)$ and note that $\chi(t_0+1,t_1+1)=\chi(t_0,t_1)$.     
Define the new function 
$$
D=\chi\partial_{12}h+(1-\chi)\delta.
$$
We have that $D\in C^1(\mathbb{R}^2)$, $D(t_1+1,t_0+1)=D(t_1,t_0)$, $D\geq\delta$ and 
\begin{equation*}
\begin{sistema}
D=\partial_{12}h \quad\mbox{ if }t_1-t_0>K \\
D=\delta \quad\mbox{ if }t_1-t_0<K-\epsilon
\end{sistema}
\end{equation*}
With a similar argument as in \cite{marobounce} we can consider the following Cauchy problem for the wave equation
\begin{equation*}
\begin{sistema}
\partial_{12}u=D(t_0,t_1)\\
u(t_0,t_0+K)=h(t_0,t_0+K)\\
(\partial_2u-\partial_1u)(t_0,t_0+K)=(\partial_2h-\partial_1h)(t_0,t_0+K).
\end{sistema}
\end{equation*}
 The solution $\tilde{h}$ is defined on $\mathbb{R}^2$, is such that $\tilde{h}\in C^2$, $\tilde{h}(\theta_1+1,\theta+1)=\tilde{h}(\theta_1,\theta)$, $\partial_{12}\tilde{h}=D\geq\delta$, $\tilde{h}=h$ for $t_1-t_0>K$ and $\tilde{h}(t_0,t_1)=-\frac{\delta}{2}(t_1-t_0)^2$ for $t_1-t_0<K-\epsilon$. Note that the function $\tilde{h}$ is the generating function of a diffeomorphism $\tilde{P}$ of the cylinder in the sense that
 \begin{equation*}
 \begin{sistema}
 E_0=\partial_1\tilde{h}(t_0,t_1) \\
 E_1=-\partial_2\tilde{h}(t_0,t_1).
 \end{sistema}
 \end{equation*}
If we define the curves $\alpha(t_0)=\partial_1\tilde{h}(t_0,t_0+K-\epsilon)$ and $\beta(t_0)=\partial_1\tilde{h}(t_0,t_0+K)$, the cylinder is divided in three regions
\begin{equation*}
\begin{split}
\Sigma_+ & =\{ (t_0,E_0): E_0>\beta(t_0)\}, \\
\Sigma_0 & =\{ (t_0,E_0): \alpha(t_0)\leq E_0\leq\beta(t_0)\},  \\
\Sigma_- & =\{ (t_0,E_0): E_0<\alpha(t_0)\}.
\end{split}
\end{equation*}  
By the special form of $\tilde{h}$ we have that $\tilde{P}$ takes the following form
\begin{equation}\label{form}
\begin{sistema}
\tilde{P}(t_0,E_0)=P(t_0,E_0)\quad\mbox{on } \Sigma_+  \\
\tilde{P}(t_0,E_0)=(t_0+\delta^{-1}E_0,E_0)\quad\mbox{on } \Sigma_-.
\end{sistema}
\end{equation}

We remember that the generating function also gives information on how to detect orbits of the corresponding diffeomorphism. By the importance of this fact we need to introduce some notation. Define, for $(a,b,c)\in \mathbb{R}^3$
$$
\tilde{\Delta}(a,b,c)=\partial_{2}\tilde{h}(a,b)+\partial_{1}\tilde{h}(b,c).
$$
One can prove that the following properties, that we will call properties \textbf{(A)}, are satisfied:
\begin{itemize}
\item[(a1)] $\tilde{\Delta}(a,b,c)$ is strictly increasing in the variables $a$ and $c$ (remembering that $\partial_{12}\tilde{h}>\delta>0$),
\item[(a2)] $\tilde{\Delta}(a,b,c)=\tilde{\Delta}(a+1,b+1,c+1)$,
\item[(a3)] $\lim_{a\to\pm\infty}\tilde{\Delta}(a,b,c)=\pm\infty$ and $\lim_{c\to\pm\infty}\tilde{\Delta}(a,b,c)=\pm\infty$
\end{itemize}
Moreover we remember that a sequence $(t_n,E_n)$ is an obit of $\tilde{P}$ if and only if
\begin{equation}\label{stat}
\tilde{\Delta}(t_{n-1},t_{n},t_{n+1})=0\quad\mbox{for every }n
\end{equation}
and $E_n=\partial_{1}\tilde{h}(t_n,t_{n+1})$ for every $n$. Sequences satisfying (\ref{stat}) are called $\tilde{h}$-stationary configurations.\\

The following proposition and corollary give bounds on the growth of a $\tilde{h}$-stationary configuration.
\begin{prop}\label{conti}
There exists $C>0$ such that if $(t_n)_{n\in\mathbb{Z}}$ is a $\tilde{h}$-stationary configuration
then
\begin{equation}\label{defc}
|\: | t_{n+1}-t_n| - |t_n-t_{n-1}|\: | \leq C.
\end{equation}
for every $n\in\mathbb{Z}$.
\end{prop}
\begin{proof}
Let $(t_n,E_n)$ be the complete orbit of the corresponding diffeomorphism $\tilde{P}$. We just have to prove that the function
$$
S(t_0,t_1)=|\:|t_{2}-t_1| - |t_1-t_{0}|\:|
$$
is bounded on $\mathbb{R}^2$. 
Note that, being $(t_n)$ the first component of an orbit we have that $t_2=t_2(t_0,E_0)=t_2(t_0,t_1)$. The diffeomorphism is defined on the cylinder, hence we have that $t_2(t_0+1,t_1+1)=t_2(t_0,t_1)+1$. This is employed to prove that
\begin{equation}\label{peri}
S(t_0+1,t_1+1)=S(t_0,t_1).
\end{equation}
If $t_1-t_0<K-\epsilon$, it comes easily from (\ref{form}) that $S(t_0,t_1)=0$. If $K-\epsilon \leq t_1-t_0 \leq K$, the periodicity (\ref{peri}) allows to consider only the compact set $\{ (t_0,t_1)\in \mathbb{R}^2:t_0\in[0,1], K-\epsilon \leq t_1-t_0 \leq K \}$ on which the continuous function $S$ is bounded by the Weierstrass theorem. Finally if $t_1-t_0>K$ we can consider the original map $P$ written in the variable $(t_0,v_0)$. The change of variable $E=\frac{1}{2}v^2$ does not alter the sequence of impact times. We have
$$
t_2=t_{1}+\frac{2}{g}v_{1}-\frac{2}{g}f[t_2,t_{1}]+\frac{2}{g}\dot{f}(t_1).
$$      
Substituting the expression of $v_1$ we have
$$
t_2=t_{1}+\frac{2}{g}v_{0}+\frac{4}{g}\dot{f}(t_1)-\frac{4}{g}f[t_1,t_{0}]+\frac{2}{g}\dot{f}(t_0)-\frac{2}{g}f[t_2,t_{1}]
$$
and, substituting the expression of $\frac{2}{g}v_{0}$, 
$$
t_2-t_1=t_1-t_0+F(t_0,t_1)
$$
where
$$
F(t_0,t_1)=\frac{4}{g}\dot{f}(t_1)-\frac{2}{g}f[t_1,t_{0}]-\frac{2}{g}f[t_2,t_{1}]
$$
is bounded. From this it is easy to conclude.
\end{proof}
\begin{cor}\label{stimm}
Let $(t_n)$ be a $\tilde{h}$-stationary configuration. Consider two non-negative integers $Q$ and $j$ such that $0\leq j<Q$. Then
$$
|t_Q-t_0|\leq |t_{j+1}-t_j|Q+\frac{Q(Q-1)}{2}C.
$$ 
\end{cor}
\begin{proof}
Before beginning the computations, note that from condition (\ref{defc}) we have
\begin{equation}\label{c1}
|t_{n+1}-t_n| \leq |t_n-t_{n-1}|+C
\end{equation}  
and
\begin{equation}\label{c2}
|t_{n}-t_{n-1}| \leq |t_{n+1}-t_{n}|+C
\end{equation}  
for every $n\in\mathbb{Z}$. Let us start. First of all we have that
\begin{equation}\label{rompi}
|t_Q-t_0|\leq \sum_{n=0}^{j-1}|t_{n+1}-t_n|+|t_{j+1}-t_j|+\sum_{n=j+1}^{Q-1}|t_{n+1}-t_n|.
\end{equation}
Let us prove by induction on $j\geq 1$ that
\begin{equation}\label{unostar}
\sum_{n=0}^{j-1}|t_{n+1}-t_n|\leq j|t_{j+1}-t_j|+\frac{j(j+1)}{2}C
\end{equation}
If $j=1$ it is true by (\ref{c2}). Moreover, we have
$$
\sum_{n=0}^{j}|t_{n+1}-t_n|=\sum_{n=0}^{j-1}|t_{n+1}-t_n|+|t_{j+1}-t_j|\leq (j+1)|t_{j+1}-t_j|+\frac{j(j+1)}{2}C
$$
where we have applied the inductive hypothesis. One can conclude easily applying (\ref{c2}) to $|t_{j+1}-t_j|$. Analogously, for every $j\geq 0$, let us prove by induction on $Q\geq j+2$ that
\begin{equation}\label{duestar}
\sum_{n=j+1}^{Q-1}|t_{n+1}-t_n|\leq (Q-j-1)|t_{j+1}-t_j|+\frac{(Q-j)(Q-j-1)}{2}C.
\end{equation}   
If $Q=j+2$ it is true by (\ref{c1}). Moreover, we have
\begin{equation}\label{rompi1}
\begin{split}
\sum_{n=j+1}^{Q}|t_{n+1}-t_n|&=\sum_{n=j+1}^{Q-1}|t_{n+1}-t_n|+|t_{Q+1}-t_Q|\\
							&\leq (Q-j-1)|t_{j+1}-t_j|+\frac{(Q-j)(Q-j-1)}{2}C+ |t_{Q+1}-t_Q|
\end{split}
\end{equation}
where we have used the inductive hypothesis. We want to estimate the quantity $|t_{Q+1}-t_Q|$ in terms of $|t_{j+1}-t_j|$. To this purpose is not difficult to prove that for every $n\in\mathbb{Z}$
\begin{equation}\label{doublestar}
|t_{n+\bar{n}+1}-t_{n+\bar{n}}|\leq |t_{n}-t_{n-1}|+(\bar{n}+1)C
\end{equation}
where $\bar{n}$ is a non-negative integer. The proof goes by induction on $\bar{n}$ using (\ref{c1}). 
Writing $Q=(j+1)+(Q-j-1)$ and applying (\ref{doublestar}) with $n=j+1$ and $\bar{n}=Q-j-1$ we get the desired estimate
$$
|t_{Q+1}-t_Q| \leq |t_{j+1}-t_j|+(Q-j)C.
$$
One can easily verify that, substituting it into (\ref{rompi1}), we conclude.\\
We proved estimates (\ref{unostar}) and (\ref{duestar}) to put them into (\ref{rompi}) and get
\begin{equation*}
\begin{split}
|t_Q-t_0|&\leq j|t_{j+1}-t_j|+\frac{j(j+1)}{2}C+|t_{j+1}-t_j|\\
         &+(Q-j-1)|t_{j+1}-t_j|+\frac{(Q-j)(Q-j-1)}{2}C \\
         &= Q|t_{j+1}-t_j|+\frac{j(j+1)}{2}C+\frac{(Q-j)(Q-j-1)}{2}C \\
         &\leq Q|t_{j+1}-t_j|+\frac{j(j+1)}{2}C+\frac{(Q+j)(Q-j-1)}{2}C \\
         &=Q|t_{j+1}-t_j|+\frac{Q(Q-1)}{2}C.
\end{split}
\end{equation*}
The proof is complete.
\end{proof}

\section{Sub-solutions and super-solutions}\label{subsuper}
The main tools of our work are introduced in the following 
\begin{defin}
A sequence $(t_n)$ is called sub-solution for a recurrence relation satisfying properties \textbf{(A)} if 
$$
\tilde{\Delta}(t_{n-1},t_n,t_{n+1})\geq 0 \quad\mbox{ for every } n\in\mathbb{Z};
$$
analogously is called a super-solution if
$$
\tilde{\Delta}(t_{n-1},t_n,t_{n+1})\leq 0 \quad\mbox{ for every } n\in\mathbb{Z}.
$$
\end{defin}

During the paper we are going to glue sub and super solutions to get new sub and super solutions. So the following lemma will be useful.
\begin{lemma}\label{glue}
Let $(s_n)_{n\in\mathbb{Z}}$ and $(t_n)_{n\in\mathbb{Z}}$ be two super-solutions (sub-solutions) for a recurrence relation $\tilde{\Delta}$ satisfying property (a1). Consider the sequence 
\begin{equation*}
\tilde{s}_n=
\begin{sistema}
s_n\quad\mbox{if } n\leq\tilde{n} \\
t_n\quad\mbox{if } n > \tilde{n}
\end{sistema}
\end{equation*}  
for some $\tilde{n}\in\mathbb{Z}$. The sequence $(\tilde{s}_n)$ is a super-solution (sub-solution) if and only if
$$
s_{\tilde{n}}\leq t_{\tilde{n}} \leq t_{\tilde{n}+1} \leq s_{\tilde{n}+1} \quad (t_{\tilde{n}}\leq s_{\tilde{n}} \leq s_{\tilde{n}+1} \leq t_{\tilde{n}+1}) 
$$
\end{lemma}
\begin{proof}
For simplicity of notation let use suppose that $\tilde{n}=0$. Let use prove it for super-solutions being the other case similar. Let us prove it by cases.\\
If $n\leq -1$ then $\tilde{\Delta}(\tilde{s}_{n-1},\tilde{s}_{n},\tilde{s}_{n+1})=\tilde{\Delta}(s_{n-1},s_{n},s_{n+1})\leq 0$.\\
If $n\geq 2$ then $\tilde{\Delta}(\tilde{s}_{n-1},\tilde{s}_{n},\tilde{s}_{n+1})=\tilde{\Delta}(t_{n-1},t_{n},t_{n+1})\leq 0$.\\  
Using the monotonicity property (a1) we can have also the following.\\ 
If $n=0$ then $\tilde{\Delta}(\tilde{s}_{-1},\tilde{s}_{0},\tilde{s}_{1})=\tilde{\Delta}(s_{-1},s_{0},t_{1})\leq \tilde{\Delta}(s_{-1},s_{0},s_{1})\leq 0$.\\
If $n=1$ then $\tilde{\Delta}(\tilde{s}_{0},\tilde{s}_{1},\tilde{s}_{2})=\tilde{\Delta}(s_{0},t_{1},t_{2})\leq \tilde{\Delta}(t_{0},t_{1},t_{2})\leq 0$.\\
form which the thesis.
\end{proof}
 
The following result by Angenent \cite[Theorem 7.1]{angenent} is the heart of our work.
\begin{teo}\label{ange}
Suppose that there exist a sub-solution ($\underline{t}_n$) and a super-solution ($\overline{t}_n$) such that
\begin{equation}\label{sub}
\limsup_{n\to +\infty}\frac{\underline{t}_n}{n}\leq\omega_0 \qquad \limsup_{n\to-\infty}\frac{\underline{t}_n}{n}\geq\omega_1
\end{equation}
and
\begin{equation}\label{super}
\liminf_{n\to-\infty}\frac{\overline{t}_n}{n}\leq\omega_0 \qquad \liminf_{n\to +\infty}\frac{\overline{t}_n}{n}\geq\omega_1
\end{equation}
for some $\omega_0<\omega_1$. 
Then there exist a compact $\mathcal{K}$ of the cylinder $\mathbb{T}\times\mathbb{R}$ and $Q\in\mathbb{N}$ such that $\tilde{P}^Q$ leaves $\mathcal{K}$ invariant and has a Bernoully shift as a factor when restricted to $\mathcal{K}$.    
\end{teo}

Note that this theorem does not give any information on the location of the set $\mathcal{K}$. It means that, in principle, the dynamics could live in the region where $P$ and $\tilde{P}$ (or equivalently $h$ and $\tilde{h}$) does not coincide and this is meaningless for our problem. Anyway we are going to choose suitable sub and super solutions that allow to reach this result.
It is important, in order to avoid confusions, to stress the fact that we are going to construct suitable sub and super solutions for the recurrent relation corresponding to the modified generating function $\tilde{h}$. To this purpose we remember that if $t_1-t_0>K$ then $\tilde{h}=h$.\\  
Let us start with the construction. It comes from the hypothesis (\ref{hpf}) that there exist $t_0^*$ and $t_0^\sharp$ such that
$$
2\dot{f}(t_0^*)=g\quad\mbox{and}\quad2\dot{f}(t_0^\sharp)=-g.
$$
Moreover, there exists $\bar{t}_0$ such that $\dot{f}(\bar{t}_0)=0$. To fix the ideas, by the periodicity of $f$ it is not restrictive to suppose that $\bar{t}_0=0$ and
\begin{equation}\label{ord}
-1\leq t_0^\sharp<0<t_0^*\leq 1.
\end{equation}
We just have to consider the translate $f(t-\bar{t}_0)$ instead of $f$. Moreover, if $t_0^\sharp$ and $t_0^*$ do not respect the order given in (\ref{ord}), by the $1$-periodicity of $f$, one can change $t_0^\sharp$ with $t_0^\sharp\pm 1$ or $t_0^*$ with $t_0^*\pm 1$ depending on the case. These changes will not let us lose the generality of the argumentation.\\  
We start constructing some orbits of $\tilde{P}$. 
\begin{prop}\label{mat}
\begin{enumerate}
\item For every integer $m$ sufficiently large the orbit of $\tilde{P}$ with initial condition $t_0=0$ and $v_0=\frac{gm}{2}$ satisfies
$$
t_{n+1}-t_n=m
$$ 
for $n\in\mathbb{Z}$. 
\item If there exists $t_0^*$ such that $2\dot{f}(t_0^*)=g$ then for every $m\in\mathbb{N}$ sufficiently large there exists an accelerating orbit $(t_n^*,v^*_n)_{n>0}$ of $\tilde{P}$ such that 
$$
t_{n+1}^*-t_n^*=m+2n
$$
Note that 
$$
t_n^*=t_0^*+n^2+(m-1)n.
$$ 
\item If there exists $t_0^\sharp$ such that $2\dot{f}(t_0^\sharp)=-g$ then for every $r\in\mathbb{N}$ sufficiently large there exists a decelerating partial orbit $(t_n^\sharp,v^\sharp_n)_{0\leq n <N}$ of $\tilde{P}$ such that
$$
t_{n+1}^\sharp-t_n^\sharp=r-2n
$$
and
$$
N=[\frac{r-K}{2}]
$$
where $[x]$ denotes the integer part of $x$.
Note that
$$
t_n^\sharp=t_0^\sharp-n^2+(r+1)n.
$$
\end{enumerate}
\end{prop}

\begin{proof}
To construct orbits of $\tilde{P}$ it is sufficient to construct orbits of $P$ and check that the condition 
\begin{equation}\label{goo}
t_{n+1}-t_n>K
\end{equation} 
is verified for every $n$. The map $P$ is expressed in the variables $(t_n,E_n)$. Performing the change of variable $E_n=\frac{1}{2}v_n^2$ we can express it in the original variables $(t_n,v_n)$ without changing the sequence of impact times. 
Now we can start the proof.
\begin{enumerate}
\item Remembering that we can suppose $\dot{f}(0)=0$, it comes from (\ref{mapp}) choosing $m$ large enough such that $gm>2\bar{v}$. 
\item As in \cite{pustsoviet} one just has to choose $t_0=t_0^*$ and $v_0=\frac{g(m-1)}{2}$ with the same condition as before.
\item Using the same idea as in \cite{pustsoviet}, we choose $t_0=t_0^\sharp$ and $v_0=\frac{g(r+1)}{2}$ with $r$ sufficiently large.  Note that to ensure that condition (\ref{goo}) is verified we can only consider a finite segment of orbit, stopping at $n=[\frac{r-K}{2}]$.  
\end{enumerate} 
\end{proof}

Now we have all the ingredients to construct the sub and the super solutions.\\ 
Let us construct first the super-solution.
By Proposition \ref{mat}, point $2$, we have the existence of an unbounded orbit $(t_n^*,E_n^*)$ such that
$$
t_n^*=t_0^*+n^2+(m-1)n
$$  
where $m$ is a positive integer. Here, $m$ is arbitrary and sufficiently large. From now on we will suppose that at least 
\begin{equation}\label{mmini}
m>\max\{\frac{2\bar{v}}{g},2K+9\}.
\end{equation}
Being $m$ sufficiently large, it comes from Proposition \ref{mat}, point $1$, that we can choose two initial energies $E_0^-$ and $E_0^+$ that give rise to two periodic orbits $(t_n^-,E_n^-)$ and $(t_n^+,E_n^+)$ such that
\begin{equation*}
\begin{split}
t_n^-&=(m+1)n\\
t_n^+&=(m+2)n
\end{split}
\end{equation*}
By the proof of Proposition \ref{mat} we have $t_0^+=t^-_0=0$.
We are going to construct the super-solution considered in Theorem \ref{ange} in which $\omega_0=m+1$, $\omega_1=m+2$ and the limits exist. To construct it, first note that    
$$
t_{0}^- \leq t^*_{0}  < t^*_{1}\leq t_{1}^-.
$$
Moreover, it is easy to see that there exists $n_+$ such that
\begin{equation}\label{wer}
t_{n_+}^* \leq t^+_{n_+}  < t^+_{n_++1}\leq t_{n_++1}^*.
\end{equation} 
It comes from the fact that $t^+_{n+1}-t_n^+=m+2$ for every $n$ while $t^*_{n+1}-t_n^*=m+2n$: let $\tilde{n}$ be the first integer such that $t^*_{\tilde{n}+1}-t_{\tilde{n}}^*>2(m+2)$. If (\ref{wer}) is not satisfied for $n_+=\tilde{n}$, it will be satisfied for $n_+=\tilde{n}+1$.   
Define the sequence $(\overline{t}_n)$ as
\begin{equation*}
\overline{t}_n=
\begin{sistema}
t^-_{n}  \quad\mbox{if  } n\leq 0   \\
t^*_n            \qquad\mbox{if  } 0<n\leq n_+ \\
t^+_{n}  \quad\mbox{if  } n>n_+.
\end{sistema}
\end{equation*}

Applying twice lemma \ref{glue} we have that $(\overline{t}_n)$ is a super-solution. 
Moreover, for large values of $n$, the sequence $(\overline{t}_n)$ coincide either with $(t^+_n)$ or $(t^-_n)$ so that $(\overline{t}_n)$ satisfies the inequalities (\ref{super}). In fact we have more: the inequalities are equalities and the limits exist. \\
Now, using the same idea let us construct the sub-solution $(\underline{t}_n)$. In Proposition \ref{mat}, point $3$, choose $r=m+3$ so that we have   
$$
t_n^\sharp=t_0^\sharp-n^2+(m+4)n
$$
and
\begin{equation}\label{diff}
t^\sharp_{n+1}-t_n^\sharp=m+3-2n.
\end{equation}
Condition (\ref{mmini}) ensure that $r$ is big enough to apply Proposition \ref{mat}.
It is easy to verify that
$$
t_{0}^\sharp \leq t^+_{0}  < t^+_{1}\leq t_{1}^\sharp.
$$
Moreover, we claim that there exists $n_-$ such that
\begin{equation}\label{otto}
t_{n_-}^- \leq t^\sharp_{n_-}  < t^\sharp_{n_-+1}\leq t_{n_-+1}^-.
\end{equation}
To prove it, note that the difference $t^\sharp_{n+1}-t_n^\sharp$ is decreasing in $n$. So let $\tilde{n}_-$ be the first integer such that $t^\sharp_{\tilde{n}_-+1}-t_{\tilde{n}_-}^\sharp<\frac{m+1}{2}$. Remembering that $t^-_{n+1}-t_n^-=m+1$, if (\ref{otto}) is not satisfied for $n_-=\tilde{n}_-$ it will be satisfied for $n_-=\tilde{n}_-+1$. Proposition \ref{mat} states that the sequence $(t_n^\sharp)$ is defined for $0\leq n< N=[\frac{m+3-K}{2}]$ (remember that $r=m+3$). So, for a good definition we have to check that $n_- < N$. We can achieve this, if, for example, 
\begin{equation}\label{sopra}
0\leq n_-< \frac{m+3-K}{2}-1=\frac{m+1-K}{2}     
\end{equation}
Combining the definition of $n_-$ with (\ref{diff}) we have that
$$
m+3-2n_-<\frac{m+1}{2}
$$
from which the condition
\begin{equation}\label{sotto}
n_->\frac{m+5}{4}
\end{equation} 
must hold.
Combining (\ref{sopra}) and (\ref{sotto}) we are done if we can find an integer $n_-$ such that
$$
\frac{m+5}{4}<n_-<\frac{m+1-K}{2}.
$$
Condition (\ref{mmini}) implies that 
$$
\frac{m+1-K}{2}-\frac{m+5}{4}>1
$$
so that we can find the requested integer $n_-$.\\   
Now it is well defined the sequence
\begin{equation*}
\underline{t}_n=
\begin{sistema}
t^+_{n}  \quad\mbox{if  } n\leq 0   \\
t^\sharp_n            \qquad\mbox{if  } 0<n\leq n_- \\
t^-_{n}  \quad\mbox{if  } n>n_-.
\end{sistema}
\end{equation*}
Applying twice lemma \ref{glue} we have that $(\underline{t}_n)$ is a sub-solution. Also in this case we have that the inequalities (\ref{sub}) are equalities and the limits exist.\\

\section{Proof of Theorem \ref{maint}}\label{prof} 
The main part of the proof of Theorem \ref{maint} consists of proving the existence of one of the sets $\mathcal{K}_i$. The existence of the entire sequence will be a simple consequences of this proof.

Applying Theorem \ref{ange} we have chaos for the modified system but we want it to be confined in the unmodified region. The sub and the super-solution we constructed depend on $m$. It is an arbitrary parameter and we will see that it can be taken in a manner such that the corresponding orbit stays in the desired region. To achieve it, we repeat the proof of Angenent \cite{angenent} of Theorem \ref{ange} modifying the parts that will lead to our aim. Angenent idea is to start with the subsolution $(\underline{t}_n)$ and the supersolution $(\overline{t}_n)$ to build two ordered sub and super solutions with certain properties. Between two ordered sub and super solutions there exists at least a $\tilde{h}$-stationary configuration and from this one can build the set $\mathcal{K}$. Now let us enter into the details.

Define the functions
\begin{equation*}
\begin{split}
z(\kappa)&=(m+1)\kappa_+-(m+2)\kappa_-  \\
Z(\kappa)&=(m+2)\kappa_+-(m+1)\kappa_-
\end{split}
\end{equation*}
where $\kappa_+=\max(\kappa,0)$ and $\kappa_-=\max(-\kappa,0)$. 
\begin{lemma}\label{emme}
There exists $M$ independent on $m$ such that for every $n$
\begin{equation*}
\begin{split}
\overline{t}_n &\geq Z(n)-M \\
\underline{t}_n &\leq z(n)+M.
\end{split}
\end{equation*}
(Here, the notation can be misleading: the sequences $(\overline{t}_n)_{n\in\mathbb{Z}}$ and $(\underline{t}_n)_{n\in\mathbb{Z}}$ depend on $m$).
\end{lemma}
\begin{proof}
First of all note that for $n\leq 0$ or $n>n_+$, the super-solution $\overline{t}_n$ and $Z(n)$ coincide, hence consider
$$
M_1>\max_{0<n\leq n_+}(Z(n)-\overline{t}_n) = \frac{9}{4}-t_0^*
$$
where we computed the maximum differentiating. Analogously we have
$$
M_2>\max_{0<n\leq n_-}(\underline{t}_n-z(n))=\frac{9}{4}+t_0^\sharp.
$$
Finally we have the thesis choosing $M=\max\{M_1,M_2\}$.
\end{proof}
Following Angenent \cite[p. 26]{angenent}, consider the new sub and super solutions
\begin{equation*}
\begin{split}
\underline{w}=&\sup\{T_{q,p}(\underline{t}):p\leq z(q)  \} \\
\overline{w}= &\inf\{T_{q,p}(\overline{t}):p\geq Z(q)  \}
\end{split}
\end{equation*}
where $(T_{q,p}(t))_n=t_{n-q}+p$. On the lines of Angenent, one can prove that for all $n$
\begin{equation}\label{frt}
\begin{split}
|\underline{w}_n-z(n)|&\leq M \\
|\overline{w}_n-Z(n)|&\leq M.
\end{split}
\end{equation}
Note that maybe that one has to increase $M$.
So far we have not defined the constant $m$ requiring only it to be large. Now we are going to fix it through the following technical lemma
\begin{lemma}\label{fondamentale}
There exist two large positive integers $m$ and $Q$ satisfying the following system of inequalities:

\begin{subequations}
   \label{sistema}
   
   \begin{eqnarray}
      \label{eq1}
      &Q>8M \\      
      \label{eq2}
      &Q(m+1)-4M>QK+\frac{Q(Q-1)}{2}C\\
      \label{eq3}
      &Q(m+2)-4M>QK+\frac{Q(Q-1)}{2}C\\
      \label{eq4}
      &QK+\frac{Q(Q-1)}{2}C>0 
   \end{eqnarray}
   \end{subequations}

where $M$ comes from (\ref{frt}), $C$ comes from Proposition \ref{conti} and $K$ comes from the definition of $\tilde{h}$.  
\end{lemma}
\begin{proof}
It is not difficult to see that it can be simply satisfied choosing $m=Q^2$ and $Q$ sufficiently large. 
\end{proof}

From now on the constants $Q$ and $m$ are fixed by lemma \ref{fondamentale}. We are going to prove that they are suitable to our purposes.\\

Let $e=(e_n)_{n\in\mathbb{Z}}$ be a sequence such that $e_n\in\{0,1\}$ for every $n\in\mathbb{Z}$. Given such a sequence define two functions $\chi_e(\kappa)$ and $\zeta_e(\kappa)$ in the following way. On the interval $nQ\leq \kappa < (n+1)Q$, with $Q$ coming from lemma \ref{fondamentale}, we define
$$
\chi_e(\kappa)=m+1+e_n
$$   
that is $m+2$ if $e_n=1$, and $m+1$ if $e_n=0$. The other function is given by
$$
\zeta_e(\kappa)=\int_0^\kappa\chi_e(s)ds.
$$
We have that for any integer $j$, $\zeta_e(jQ)$ is an integer so that we can define the new sub and super solutions
\begin{equation*}
\begin{split}
\underline{x}_e=&\sup\{T_{jQ,\zeta_e(jQ)}(\underline{w}):j\in\mathbb{Z}  \} \\
\overline{x}_e= &\inf\{T_{jQ,\zeta_e(jQ)}(\overline{w}):j\in\mathbb{Z}  \}.
\end{split}
\end{equation*}
One can prove that for all $n$
\begin{equation}\label{estizeta}
\begin{split}
|\underline{x}_{e,n}-\zeta_e(n)|&\leq M \\
|\overline{x}_{e,n}-\zeta_e(n)|&\leq M.
\end{split}
\end{equation}
The proof of this fact is the same as in \cite[Proposition 7.2]{angenent}. In our particular case the constants $\rho_0$ and $\rho_1$ of Angenent take the values $\rho_0=m+1$ and $\rho_1=m+2$.\\
It is easy to see that $\underline{x}_e-M$ and $\overline{x}_e+M$ are sub and super solutions such that $\underline{x}_e-M\leq \overline{x}_e+M$. So, using \cite[Theorem 4.2]{angenent}, there exists at least one $\tilde{h}$-stationary configuration $x^*=(x^*_n)$ such that for every $n$,
\begin{equation}\label{pizz}
\underline{x}_{e,n}-M\leq x^*_n \leq \overline{x}_{e,n}+M.
\end{equation}
From the $\tilde{h}$-stationary configuration $x^*$ we can construct an orbit of the modified map $\tilde{P}$. Anyway, we are going to prove that every $\tilde{h}$-stationary configuration $x$ satisfying condition (\ref{pizz}) is such that  $x_{n+1}-x_{n}>K$ for every $n$ so that it is an $h$-stationary configuration as well. From this latter configuration one can construct an orbit of the original map $P$. In fact we have

\begin{lemma}\label{dove}  
Let $x$ be a $\tilde{h}$-stationary configuration satisfying condition (\ref{pizz}). We have that
$$
x_{n+1}-x_{n}>K
$$
for every $n\in\mathbb{Z}$. 
\end{lemma}
\begin{proof}
Suppose by contradiction that there exists $j$ such that $x_{j+1}-x_{j}\leq K$. First of all there exists $\bar{n}$ such that $\bar{n}Q\leq j < (\bar{n}+1)Q$ but, for simplicity of notation we will suppose that $\bar{n}=0$. Applying Corollary \ref{stimm} we have that the condition
\begin{equation}\label{contra}
|x_Q-x_0|\leq KQ+\frac{Q(Q-1)}{2}C
\end{equation} 
must hold. We will see that the choice of the constants coming from lemma \ref{fondamentale} prevents condition (\ref{contra}) from being satisfied. In fact we are going to prove that the condition 
\begin{equation}\label{tre}
x_{(n+1)Q}-x_{nQ}>QK+\frac{Q(Q-1)}{2}C
\end{equation}
holds for every $n\in\mathbb{Z}$.    
To prove it we need the following estimate on the sub and the super-solution that holds for every $n$
$$
[\underline{x}_{e,(n+1)Q}-M] - [\overline{x}_{e,nQ}+M]>QK+\frac{Q(Q-1)}{2}C
$$
It follows from (\ref{estizeta}) and the definition of $\zeta_e$, indeed
\begin{equation*}
\begin{split}
(\underline{x}_{e,(n+1)Q}-M) - (\overline{x}_{e,nQ}+M) & \geq (\zeta_e((n+1)Q)-2M)-(\zeta_e(nQ)+2M) \\ 
& = \zeta_e((n+1)Q)-\zeta_e(nQ)-4M \\
& = \int_{nQ}^{(n+1)Q}(m+1+ e_n) ds-4M\\
& =(m+1+ e_n)Q -4M \\
& > QK+\frac{Q(Q-1)}{2}C
\end{split}
\end{equation*}
with the last inequality coming from inequalities (\ref{eq2}) and (\ref{eq3}) in lemma \ref{fondamentale}. Now it is easy to prove (\ref{tre}):
\begin{equation*}
\begin{split}
x_{(n+1)Q}-x_{nQ} & =[x_{(n+1)Q}-(\underline{x}_{e,(n+1)Q}-M)] + [(\underline{x}_{e,(n+1)Q}-M)-(\overline{x}_{e,nQ}+M)]  \\
& + [(\overline{x}_{e,nQ}+M)-x_{nQ}]\\
& >0+QK+\frac{Q(Q-1)}{2}C+0\\
& =QK+\frac{Q(Q-1)}{2}C.
\end{split}
\end{equation*}
This gives the contradiction with (\ref{contra}).
\end{proof}
This lemma allows to affirm that every $\tilde{h}$-stationary configuration satisfying condition (\ref{pizz}) is in fact a $h$-stationary configuration. 
Now we are ready to construct the compact set $\mathcal{K}$.\\
Let $\Sigma_e$ be the set of all $h$-stationary configurations $\tau=(\tau_n)$ satisfying condition (\ref{pizz}).
We have just seen that it is non-empty. Let $\Sigma$ be the union of all the $\Sigma_e$ where $e$ ranges over all possible sequences of $\{0,1\}$. Then we put
$$
\mathcal{C}_e=\{(\tau_0,\tau_1) \mod \mathbb{Z}:\tau\in\Sigma_e \}
$$ 
and
$$
\mathcal{C}=\{(\tau_0,\tau_1) \mod \mathbb{Z}:\tau\in\Sigma \}
$$
Once again following Angenent \cite[Proposition 7.3]{angenent}, we can prove that these two sets are compact. Then we can define 
$$
\mathcal{K}_e=\{(\tau_0,\partial_{\tau_0}h(\tau_0,\tau_1)) :(\tau_0,\tau_1)\in\mathcal{C}_e \}
$$
and analogously $\mathcal{K}$. Note that if we call $E_0=\partial_{\tau_0}h(\tau_0,\tau_1)$ then the set $\mathcal{K}$ can be seen as a subset of initial conditions for the map $P$. Being the sets $\mathcal{C}_e$ and $\mathcal{C}$ compact, also the sets $\mathcal{K}_e$ and $\mathcal{K}$ are compact. To conclude the proof of Theorem \ref{maint} we need to prove that $\mathcal{K}$ is invariant under $P^Q$ and find a continuous function that gives the semi-conjugacy of the map $P^Q$ restricted to the set $\mathcal{K}$ with the Bernoulli shift. \\
Consider a point $(\tau_0,E_0)\in\mathcal{K}$. By the definition of the set $\mathcal{K}$, we have that $(\tau_0,E_0)\in\mathcal{K}_e$ for some sequence $(e_n)$. From this procedure, one can define a map
 $$
 \varepsilon:\mathcal{K}\rightarrow \{0,1\}^\mathbb{Z}.
 $$   
 However, two different sets $\mathcal{K}_{e^1}$ and $\mathcal{K}_{e^2}$ could overlap and the map $\varepsilon$ could be multivalued. We are going to prove that this cannot occur. The proof is inspired by Angenent \cite[Proposition 7.4]{angenent} but we prefer to perform it to put in evidence some details.
 \begin{lemma}
 The sets $\mathcal{K}_e$ are pairwise disjoint.
 \end{lemma}
 \begin{proof}
Being $\partial_{12}h\geq\delta>0$ we just need to prove that the sets $\mathcal{C}_e$ are pairwise disjoint. Let us prove it. A point $(\tau_0,\tau_1)\in \mathcal{C}_e$ corresponds to a $h$-stationary configuration $\tau$. Using (\ref{pizz}) and (\ref{estizeta}) it is not difficult to see that
$$
|(\tau_{(n+1)Q}-\tau_{nQ})-(\zeta_e((n+1)Q)-\zeta_e(nQ))| \leq 4M.
$$
To get it consider the quantity $(\tau_{(n+1)Q}-\tau_{nQ})$, estimate it from both side using (\ref{pizz}) and then use (\ref{estizeta}). By the definition of the function $\zeta(\kappa)$ we have that  
\begin{equation}\label{fr}
|(\tau_{(n+1)Q}-\tau_{nQ}) - (m+1+ e_n)Q |\leq 4M.
\end{equation}
Consider now two different sequences of $\{0,1\}$ and call them $e^1=(e^1_n)$ and $e^2=(e_n^2)$. Being different there must exists $\bar{n}\in\mathbb{Z}$ such that, for example $e^1_{\bar{n}}=0$ and $e_{\bar{n}}^2=1$. Suppose by contradiction that $\mathcal{C}_{e^1}\cap\mathcal{C}_{e^2}=(\tau_0,\tau_1)$. The point $(\tau_0,\tau_1)$ generates a $h$-stationary configuration $(\tau_n)$ that must satisfies condition (\ref{fr}) for both $e^1$ and $e^2$. Hence
\begin{equation}\label{fre}
\begin{split}
|(\tau_{(\bar{n}+1)Q}-\tau_{\bar{n}Q}) & - (m+1)Q |\leq 4M \\
|(\tau_{(\bar{n}+1)Q}-\tau_{\bar{n}Q}) & - (m+2)Q |\leq 4M.
\end{split}
\end{equation}    
Remembering that we imposed (\ref{eq1}) we have that 
\begin{equation*}
\begin{split}
8M & <Q=(m+2)Q-(m+1)Q+(\tau_{(\bar{n}+1)Q}-\tau_{\bar{n}Q})-(\tau_{(\bar{n}+1)Q}-\tau_{\bar{n}Q}) \\
   & \leq |(\tau_{(\bar{n}+1)Q}-\tau_{\bar{n}Q}) - (m+1)Q |+|(\tau_{(\bar{n}+1)Q}-\tau_{\bar{n}Q}) - (m+1)Q |\leq 8M
\end{split}
\end{equation*}    
from (\ref{fre}). This contradiction concludes the proof.  
\end{proof}
It comes from the proof the precise definition of the map $\varepsilon$, we have:
$$
(\varepsilon(\tau_0,E_0))_n=e_n
$$
where $e_n$ is the only one that satisfies (\ref{fr}). Moreover, if we equip $\{0,1\}^\mathbb{Z}$ with the product topology the function is continuous. From the fact that we can define the functions $\chi_e$ and $\zeta_e$ for every sequence $e$, it is not difficult to see that the map $\varepsilon$ is surjective. Finally, we remember that the Bernoulli shift $\sigma:\{0,1\}^\mathbb{Z}\rightarrow \{0,1\}^\mathbb{Z}$ acts by $\sigma(e)_n=e_{n+1}$. Using (\ref{fr}) we have that the map $P^Q$ maps the set $\mathcal{K}_e$ into $\mathcal{K}_{\sigma(e)}$ from which we have that the set $\mathcal{K}$ is invariant under $P^Q$. Moreover it is easily seen that the following diagram
$$
\begin{diagram}
\node{\mathcal{K}} \arrow{e,t}{P^Q} \arrow{s,l}{\varepsilon}
      \node{\mathcal{K}} \arrow{s,r}{\varepsilon} \\
\node{\{0,1\}^\mathbb{Z}} \arrow{e,b}{\sigma}
   \node{\{0,1\}^\mathbb{Z}}
\end{diagram}
$$   
commutes.\\

We conclude proving that our method allows to find infinitely many disjoint invariant set $\mathcal{K}$ leading to chaos. To this purpose, it is sufficient to prove that one can detect several sets $\mathcal{C}$. The set $\mathcal{C}$ is determined by sequences satisfying condition (\ref{pizz}). Using (\ref{estizeta}) and the definition of the function $\zeta_e$ one can prove that for every $n\in\mathbb{Z}$ the estimates 
\begin{equation*}
\begin{split}
\overline{x}_{e,n}+M & \leq n(m+2)+2M \\
\underline{x}_{e,n}-M & \geq n(m+1)-2M
\end{split}
\end{equation*}   
holds. So, every configuration $(x^*_n)$ defining the set $\mathcal{C}$ is such that 
$$
n(m+1)-2M \leq x^*_n\leq n(m+2)+2M
$$ 
for every $n$. Now it is easy to conclude letting $m$ (or equivalently $Q$) tend to $+\infty$.

\bibliographystyle{plain}
\bibliography{biblio4}
\end{document}